\documentclass[11pt,oneside]{amsart}
\pdfoutput=1 

\numberwithin{equation}{section}

\usepackage{xcolor}
\definecolor{cadmiumgreen}{rgb}{0.0, 0.42, 0.24}
\usepackage[
colorlinks, citecolor=cadmiumgreen,
pagebackref,
pdfauthor={Andreas Gross, Farbod Shokrieh}, 
pdftitle={Cycles, cocycles, and duality on tropical manifolds},
pdfstartview ={FitV},
]{hyperref}

\usepackage[
alphabetic,
backrefs,
msc-links,
nobysame,
lite,
]{amsrefs}

\usepackage[text width=360pt, text height=615pt]{geometry}

\usepackage[T1]{fontenc}
\usepackage[utf8]{inputenc}

\usepackage{mathptmx} 
\usepackage[scaled]{helvet}
\usepackage{courier}
\usepackage{microtype}
\usepackage{amssymb,amsthm, amsmath,mathtools}	

\usepackage{dsfont}						
\usepackage{tikz}							
\usetikzlibrary[matrix,arrows,calc]
\usepackage{enumitem}

\newcommand{\xqedhere}[2]{
  \rlap{\hbox to#1{\hfil\llap{\ensuremath{#2}}}}}

 \theoremstyle{plain}
 \newtheorem{thm}{Theorem}[section]
  \theoremstyle{definition}
  \newtheorem{defn}[thm]{Definition}
  \theoremstyle{definition}
  
  \theoremstyle{plain}
  \newtheorem{lem}[thm]{Lemma}
  \theoremstyle{plain}
  \newtheorem{cor}[thm]{Corollary}
  \theoremstyle{remark}
  \newtheorem{rem}[thm]{Remark}
  \theoremstyle{plain}
  \newtheorem{prop}[thm]{Proposition}
  \theoremstyle{plain}
  
  \theoremstyle{remark}
  
  \theoremstyle{remark}
  
  \theoremstyle{plain}
  
  \theoremstyle{definition}
  
  \numberwithin{equation}{section}

\DeclareMathOperator{\LC}{LC}
\DeclareMathOperator{\PP}{PP}
\DeclareMathOperator{\LPP}{LPP}
\DeclareMathOperator{\Hom}{Hom}

\DeclareMathOperator{\Star}{S}

\DeclareMathOperator{\relint}{relint}

\DeclareMathOperator{\Lin}{Span}
\DeclareMathOperator{\lineal}{lineal}

\DeclareMathOperator{\Mink}{MW}

\DeclareMathOperator{\Aff}{Aff}

\renewcommand{\injlim}{\varinjlim}

\newcommand{\R}{{\mathds R}}

\newcommand{\Q}{{\mathds Q}}
\newcommand{\Z}{{\mathds Z}}
\newcommand{\N}{{\mathds N}}
\newcommand{\M}{{\mathrm M}}

\newcommand{\mC}{{\mathcal C}}
\newcommand{\mZ}{{\mathcal Z}}
\newcommand{\mPP}{{\mathcal{PP}}}
\newcommand{\mLPP}{{\mathcal{LPP}}}

\usepackage[normalem]{ulem}

\begin{document}

\title{Cycles, cocycles, and duality on tropical manifolds}

\author{Andreas Gross}
\address{Colorado State University \\ Fort Collins, CO 80523}
\email{\href{mailto:andreas.gross@colostate.edu}{andreas.gross@colostate.edu}}

\author{Farbod Shokrieh}
\address{University of Washington \\ Seattle, WA 98195}
\email{\href{mailto:farbod@uw.edu}{farbod@uw.edu}}



\subjclass[2020]{
\href{https://mathscinet.ams.org/msc/msc2020.html?t=14T10}{14T10},
\href{https://mathscinet.ams.org/msc/msc2020.html?t=14C17}{14C17},
\href{https://mathscinet.ams.org/msc/msc2020.html?t=05B35}{05B35},
\href{https://mathscinet.ams.org/msc/msc2020.html?t=52B99}{52B99}
}


\begin{abstract}
We prove a Poincar\'e duality for the Chow rings of smooth fans whose support are tropical linear spaces. As a consequence, we show that cycles and cocycles on tropical manifolds are Poincar\'e dual to each other. This allows us to define pull-backs of tropical cycles along arbitrary morphisms with smooth target.
\end{abstract}

\maketitle


\section{Introduction}
\renewcommand*{\thethm}{\Alph{thm}}

Much of the development of tropical geometry and tropical intersection theory has been based on two theorems proven in \cite{FMSS95}. Theorem~1 in \cite{FMSS95} gives an explicit expression of the Chow ring $A^*(X_{\Sigma})$ of a smooth toric variety $X_{\Sigma}$ in term of its fan $\Sigma$, and Theorem~3 in \cite{FMSS95} shows that $A^*(X_{\Sigma})$ is a Poincar\'e duality ring of dimension $\dim X_{\Sigma}$ if $X_{\Sigma}$ is complete. Together, these theorems have led to the description of $A^*(X_{\Sigma})$ via Minkowski weights given in \cite{FS97} that has been vital in the definition of the tropicalization map \cite{Tev07,ST08,Katz09} and in tropical intersection theory \cite{FS97,AR10,Katz12}. 

If $X_{\Sigma}$ is not complete, the combinatorial expression for $A^*(X_{\Sigma})$ in \cite{FMSS95} still applies, but it is generally not a Poincar\'e duality ring and cannot be described by tropical cycles on $\Sigma$. 

Our first main result states that, in the case where the support $\vert\Sigma\vert$ of $\Sigma$ is a tropical linear space, the ring $A^*(X_{\Sigma})$ is a Poincar\'e duality ring of dimension $\dim(\Sigma)$. 

\begin{thm}[= Theorem {\ref{thm:duality for fans2}}]
\label{thm:duality for fans}
Let $\Sigma$ be a smooth fan in $N_\R$ such that $|\Sigma|$ is a tropical linear space. Then the Chow ring $A^*(X_{\Sigma})$ of its associated toric variety $X_\Sigma$ is a Poincar\'e duality ring.
\end{thm}

A similar result is proven for the {\em rational} Chow ring  in \cite[Theorem 3.5]{Adiprasito}, where the proof relies on the topological properties of the tropical linear space $\vert\Sigma\vert$. For our application in tropical intersection theory, the rational case is not enough, since tropical cycles have integer weights. The case of integer coefficient is at the heart of the combinatorial Hodge theory developed in \cite{AHK15}, where it has been shown that $A^*(X_{\Sigma})$ has Poincar\'e duality in the special case that $\Sigma$ is the `fine subdivision' on a tropical linear space (see also \cite{bes} for a non-inductive proof). Our proof of Theorem~\ref{thm:duality for fans} uses, among other things, this special case as well as the weak factorization theorem for toric varieties.

\medskip

A different description of $A^*(X_{\Sigma})$ in terms of piecewise polynomial functions has been given in \cite{Brion96}. This has led to the concept of tropical cocycles in \cite{KP08,F11}. For every boundaryless rational polyhedral space $X$, the ring of tropical cocycles $C^*(X)$ acts on the group of tropical cycles $Z_*(X)$ via the intersection pairing
\[
C^*(X)\times Z_*(X)\to Z_*(X), \;\; (c,\alpha)\mapsto c\cdot \alpha \ .
\]

If $X$ is a tropical manifold, then $Z_*(X)$ can also be endowed with a ring structure using the tropical intersection product (see \cite{FR10,Shaw13}). In this case, it has long been expected that intersecting tropical cocycles with the {\em fundamental cycle} $[X]$, the unity in $Z_*(X)$, defines an isomorphism of rings. Our second main results confirms this expectation. 

\begin{thm}[= Theorem \ref{thm:duality for manifolds2}]
\label{thm:duality for manifolds}
Let $X$ be a boundaryless tropical manifold. Then the morphism
\[
C^*(X)\to Z_*(X),\;\; \alpha\mapsto \alpha\cdot[X] 
\]
is a an isomorphism of rings.
\end{thm}

In \cite{F11} this was shown to be true in some special cases: when the underlying space of $X$ is a topological manifold, and without additional assumptions on $X$ in the codimension-$1$ and codimension-$\dim X$ graded pieces of the rings.

Note that our exposition of tropical cocycles in \S\ref{subsec:cocycles on manifolds} also fixes the inaccuracies 	in their development in \cite{F11} (see Remark \ref{rem:inconsitence in definition of cocycles}).

The isomorphism of Theorem \ref{thm:duality for manifolds} is an important step in the development of tropical intersection theory. Namely, given a morphism $X\to Y$ from a rational polyhedral space $X$ to a tropical manifold $Y$ (both without boundary) and a cycle $A\in Z_*(X)$, one can use cocycles to canonically define a pull-back $Z_*(X)\to Z_*(\vert A\vert)$. Previously, the existence of the pull-back was only known in the special case where both $X$ and $Y$ are tropical manifolds \cite{FR10}.

\begin{cor}[= Corollary \ref{cor:pullbacks of tropical cycles}]
\label{cor:pullbacks of tropical cycles2}
Let $f\colon X\to Y$ be a morphism from a boundaryless rational polyhedral space $X$ into a boundaryless tropical manifold $Y$, and let $A\in Z_*(X)$. Then there exists a unique pull-back morphism
\[
f^*\colon Z_*(Y)\to Z_*(X)
\]
that satisfies $f^*[Y]=A$ and
\[
f^*(c\cdot B)= f^*(c)\cdot f^*(B)
\]
for every $c\in C^*(Y)$ and $B\in Z_*(Y)$.
\end{cor}

We should also note that Theorem \ref{thm:duality for manifolds} can be used for an alternative definition of the tropical intersection product, replacing the tropical modifications (as in \cite{Shaw13}) and the diagonal argument (as in \cite{FR10}) by the use of tropical cocycles.

\medskip

\noindent {\em Related work. } 
After our paper appeared on the arXiv, Amini and Piquerez announced in \cite{AP19} a proof of Theorem \ref{thm:duality for fans} via a comparison with with the tropical homology groups of compactified tropical linear spaces. With this in mind, the duality we prove is closely related to Poincaré duality  for tropical homology and cohomology on tropical manifolds  \cite{jss,jrs,GSH}. It has now also been shown by Ardila, Denham, and Huh in \cite[Section 6]{adh} that Lefschetz properties, which include Poincaré duality, are intrinsic to the support of the fan.

\section{Preliminaries}
\renewcommand*{\thethm}{\arabic{section}.\arabic{thm}}
\subsection*{Notation}
Throughout this paper, $N$ will denote a lattice and $M=\Hom(N,\Z)$ will denote its dual.  The symbol $\Sigma$ will always denote a smooth rational polyhedral fan in $N_\R=N\otimes_\Z \R$. We denote the set of its $k$-dimensional cones by $\Sigma(k)$, and for each of its rays $\rho\in \Sigma(1)$ we denote its primitive generator by $u_\rho\in N$. The evaluation pairing $M\times N\to \Z$ will be denoted by $\langle \cdot , \cdot\rangle$. 

\subsection{Conical sets and fans}
We refer to \cite{CLS} for background on rational polyhedral cones and fans.

A \emph{conical rational polyhedral set} in $N_\R$ is a union of finitely many rational polyhedral cones.
The \emph{lineality space} $\lineal(P)$ of a conical rational polyhedral set $P\subseteq N_\R$ is the set of all $w\in N_\R$ such that $P+\lambda w \subseteq P$ for all $\lambda\in \R$.
An isomorphism between two conical rational polyhedral sets $P\subseteq N_\R$ and $Q\subseteq N'_\R$ is a map $P\to Q$ that can be extended to an integral linear isomorphism $\Lin(P)\to \Lin(Q)$.
A \emph{fan structure} on a conical rational polyhedral set $P$ in $N_\R$ is a rational polyhedral fan $\Sigma$ in $N_\R$ whose support $\vert\Sigma\vert$ is $P$. 
The \emph{local cone} of a conical rational polyhedral set $P$ at a point $p\in P$ is given by the set
\[
\LC_{p}(P)=\{w\in N_\R \mid p+\epsilon w \in P \text{ for all }\epsilon \text{ small enough } \} \ .
\]
This is a conical rational polyhedral set in $N_\R$ again, and it only depends on an open neighborhood of $p$ in $P$.

A rational polyhedral fan $\Sigma$ in $N_\R$ is \emph{smooth} if every cone $\sigma\in \Sigma$ is the positive hull of elements $u_1\ldots, u_k\in N$ that can be completed to a basis of $N$. 
If $\Sigma$ is a rational polyhedral fan in $N_\R$ and $\sigma\in \Sigma$, then the \emph{star of $\Sigma$ at $\sigma$}, denoted by $\Star_\Sigma(\sigma)$, is the rational polyhedral fan in $N_\R/\Lin(\sigma)= (N/(\Lin(\sigma)\cap N))_\R$ whose cones are the images under the projection $N_\R\to N_\R/\Lin(\sigma)$ of the cones of $\Sigma$ containing $\sigma$.
If $\sigma$ is a cone of a smooth rational polyhedral fan $\Sigma$ in $N_\R$, and $u_1,\ldots,u_k$ are the primitive generators of the rays of $\sigma$, then then the \emph{stellar subdivision} $\Sigma^{\star}(\sigma)$ of $\Sigma$ at $\sigma$ is the unique smooth rational polyhedral fan in $N_\R$ that refines $\Sigma$ and has exactly one ray more than $\Sigma$, namely the ray spanned by $\sum_{i=1}^k u_i$. 

\subsection{Tropical linear spaces}

We recall the definition of tropical linear spaces following \cite[\S 4.2]{TropBook}.
Let $\M$ be a loop-free matroid on the ground set $E=E(\M)$. For each circuit $C\subseteq E$ of $\M$ we define a subset $H_C$ of $\R^E$ by
\[
H_C\coloneqq \{ w\in \R^E\mid \min\{w_e\mid e\in C\}\text{ is attained at least twice}\} \ ,
\]
where $w_e$ denotes the coordinate of $w$ corresponding to $e\in E$. We denote by $\widetilde L_M$ the intersection of all sets $H_C$, where $C$ ranges over all circuits of $\M$. It is a conical rational polyhedral set in $\R^E$. The all-one vector $\mathbf 1 \in \R^E$ is contained in the lineality space $\lineal(\widetilde L_\M)$. Therefore, the quotient
\[
L_\M\coloneqq \widetilde L_\M /\R \mathbf 1
\]
is a conical rational polyhedral set in $\R^M/\R\mathbf 1$. Any conical rational polyhedral set $P$ isomorphic to $L_\M$, for some loop-free matroid $\M$, is called a \emph{tropical linear space}. There are several important fan structures on $L_\M$, one of them being the \emph{fine subdivision} introduced in \cite{FeiStu} (see also \cite[\S 3]{AK06}), which we will denote by $\Sigma_\M$.

\subsection{Chow rings of fans and toric varieties} 
\label{subsec:chow rings of fans}
For a smooth fan $\Sigma$ in $N_\R$ we consider the commutative graded ring
\[
S^*(\Sigma)=\Z\left[\{x_\rho\}_{\rho\in \Sigma(1)}\right]
\]
freely generated by the rays of $\Sigma$.  Here, the monomials $x_\rho$ for $\rho\in\Sigma(1)$ have degree $1$.
Let $\mathcal I_\Sigma\subseteq S(\Sigma)$ be the ideal generated by the square-free monomials $\prod_{k\in K}x_k$ for all subsets $K\subseteq \Sigma(1)$ that do not span a cone of $\Sigma$. Let $\mathcal J_\Sigma\subseteq S(\Sigma)$ be the ideal generated by the elements 
\[
\sum_{\rho\in\Sigma(1)} \langle m, u_\rho \rangle x_\rho \  ,
\]
for all $m\in M$. The quotient
\[
\mathrm{SR}^*(\Sigma)\coloneqq S^*(\Sigma)/\mathcal I_\Sigma
\]
is called the \emph{Stanley-Reisner ring} of $\Sigma$, and the quotient
\[
R^*(\Sigma)=S^*(\Sigma)/(\mathcal I_\Sigma +\mathcal J_\Sigma)
\]
is called the \emph{Jurkiewicz-Danilov ring} or the \emph{Chow ring} of $\Sigma$. As implied in the notation, we will consider both $\mathrm{SR}^*(\Sigma)$ and $R^*(\Sigma)$ as $\N$-graded rings, with grading induced from $S^*(\Sigma)$.

Let $X_\Sigma$ denote the smooth toric variety (over an algebraically closed field) associated to $\Sigma$, and let $A^*(X_\Sigma)$ denote its Chow ring. Consider the ring homomorphism
\[
\begin{aligned}
S^*(\Sigma) &\to A^*(X_\Sigma) \\
x_\rho &\mapsto [D_\rho] \ ,
\end{aligned}
\]
where $[D_\rho]$ denotes the torus-invariant Weil divisor class associated to $\rho\in\Sigma(1)$.
The induced graded morphism
\[
R^*(\Sigma)\to A^*(X_\Sigma) \ .
\]
 is an isomorphism by \cite[{\S 3.1}]{Brion96}. For every $\sigma\in \Sigma$, the monomial $x_\sigma=\prod_{\rho\in\sigma(1)}x_\rho$ is sent to the class $[V(\sigma)]$ associated to the closure of the orbit corresponding to $\sigma$ via the orbit-cone correspondence. 

The description of $A^k(X_\Sigma)$ in \cite{FMSS95} implies that the morphism
\begin{equation}
\label{equ:presentation of R}
\Z^{\Sigma(k)}\to R^k(\Sigma) \ ,\;\; (a_\sigma)_{\sigma\in\Sigma(k)}\mapsto \sum_{\sigma\in\Sigma(k)} a_\sigma x_\sigma
\end{equation}
is surjective, with kernel generated by the relations 
\[
\sum_{\sigma\in\Sigma(k), \tau\subset \sigma} \langle m ,  u_{\sigma/\tau}\rangle x_\sigma
\]
for all $\tau\in \Sigma(k-1)$ and $m\in \tau^\perp$. Here $u_{\sigma/\tau}$ denotes any representative of the lattice normal vector of $\sigma$ with respect to $\tau$, i.e. any representative of the primitive generator of the ray $\sigma/\tau$. In particular, the morphism in (\ref{equ:presentation of R}) induces an injection 
\[\Hom(R^k(\Sigma),\Z)\to (\Z^{\Sigma(k)})^\vee\cong \Z^{\Sigma(k)}\] 
onto the group $\Mink_k(\Sigma)$ of $k$-dimensional \emph{Minkowski weights} as defined in \cite[\S 3]{FS97}. 

\begin{rem}
In \cite{Brion96} it is only stated that $R^*(\Sigma)_\Q\to A^*(X_\Sigma)_\Q$ is an isomorphism, but the proof given there actually works over the integers as well when one assumes that $\Sigma$ is smooth, as opposed to only simplicial. 
\end{rem}

The Stanley-Reisner ring $\mathrm{SR}^*(\Sigma)$ can also be described via piecewise polynomial functions. Let $\PP^*_\Z(\Sigma)$ denote the ring of all continuous functions $\phi\colon|\Sigma|\to \R$ such that, for every $\sigma\in\Sigma$, the restriction $\phi|_\sigma$ is given by a polynomial function {\em with integer coefficients}. Naturally $\PP^*_\Z(\Sigma)$ is graded by the degree of the polynomials functions. For every ray $\rho\in\Sigma(1)$ there is a unique function $\phi_\rho\in \PP_\Z^1(\Sigma)$, called \emph{Courant-function}, that has value $1$ on $u_\rho$ and vanishes on all rays $\rho\neq\rho'\in\Sigma(1)$. It is shown in \cite[\S 1.3]{Brion96} that sending $x_\rho\in S(\Sigma)$ to $\phi_\rho$ induces an isomorphism
\[
\mathrm{SR}^*(\Sigma)\to \PP^*_\Z(\Sigma) \ .
\]
The elements of the ideal $\mathcal J_\Sigma$ are mapped onto the linear functions in $\PP^1_\Z(\Sigma)$ so, if $\LPP_\Z(\Sigma)$ denotes the ideal generated by linear functions, we obtain an isomorphism
\[
R^*(\Sigma)\to \PP^*_\Z(\Sigma)/\LPP_\Z(\Sigma) \ .
\]

\begin{rem}
It is shown in \cite{Brion96} that the morphism $R^*(\Sigma)\to \PP^*_\Z(\Sigma)$ described above is an isomorphism when tensored with $\Q$, and $\Sigma$ is only assumed to be simplicial. In our situation, where $\Sigma$ is smooth, Brion's proof still works after replacing the rationals by the integers.
\end{rem}

\section{Poincar\'e duality on toric varieties}

\subsection{Lineality spaces of tropical linear spaces}

We need some basic results about tropical linear spaces. As we could not locate them in the literature, we will give brief proofs. 

Following \cite[\S 7.1]{OX}, the \emph{parallel connection} $P(\M_1,\M_2)$ of two matroids $\M_1$ and $\M_2$ with respect to two basepoints $p_1\in E(\M_1)$ and $p_2\in E(\M_2)$ is a matroid on $E(P(\M_1,\M_2))=E(\M_1)\cup E(\M_2)/(p_1\sim p_2)$. For simplicity we will assume that $p\coloneqq p_1=p_2$ and $E(\M_1)\cap E(\M_2)=\{p\}$, in which case we can take $E(P(\M_1,\M_2))=E(\M_1)\cup E(\M_2)$. Then $P(\M_1,\M_2)$ has precisely the following three types of circuits:
\begin{enumerate}[label=(\arabic*)]
\item All circuits of $\M_1$,
\item All circuits of $\M_2$,
\item $(C_1\cup C_2) \setminus\{p\}$ for all pairs of circuits $C_1$ of $\M_1$ and $C_2$ of $\M_2$ such that $p\in C_1\cap C_2$. 
\end{enumerate}

\begin{lem}
\label{lem:parallel connection}
Let $\M_1$ and $\M_2$ be loop-free matroids, and let $p_1\in E(\M_1)\eqqcolon E_1$ and $p_2\in E(\M_2)\eqqcolon E_2$. Let $\M=P(\M_1,\M_2)$ denote the parallel connection of $\M_1$ and $\M_2$ with respect to the basepoints $p_1$ and $p_2$. Then $L_\M$ is isomorphic to the quotient of $L_{\M_1\oplus \M_2}$ by a $1$-dimensional subspace of its lineality space. In particular, $L_\M\cong L_{\M_1}\times L_{\M_2}$.
\end{lem}

\begin{proof}
We first observe that for every conical rational polyhedral set $P$ we have $P\cong \left(P/\lineal(P) \right)\times \lineal(P)$. Therefore, if $L$ and $L'$ are subspaces of $\lineal(P)$ of the same dimension, we have $P/L\cong P/L'$. Since $\widetilde L_{\M_1\oplus \M_2}\cong\widetilde L_{\M_1}\times \widetilde L_{\M_2}$,  by what we just observed, it suffices construct an isomorphism $\widetilde L_\M\cong \widetilde L_{\M_1}\times \widetilde L_{\M_2}/\R(v_1,v_2)$, with $(v_1,v_2)\in \lineal(L_{\M_1})\times \lineal (L_{\M_2})\setminus \R(\mathbf 1,\mathbf 1)$, that maps $\mathbf 1$ to the class of $(\mathbf 1,\mathbf 1)$. 

To prove this, first assume that either $p_1$ is a coloop of $\M_1$ or $p_2$ is a coloop of $\M_2$. Without loss of generality we may assume the former. By definition, $\M=(\M_1\setminus p_1)\oplus \M_2$. In this case, the $p_1$-th standard basis vector $e_{p_1}\in \R^{E_1}$ is contained in $\lineal(\widetilde L_{\M_1})$. Therefore, we have a chain of natural isomorphisms
\[
\widetilde L_\M = \widetilde L_{(\M_1\setminus p_1)\oplus \M_2}\cong \widetilde L_{\M_1\setminus p_1}\times \widetilde L_{\M_2} \cong \big(\widetilde L_{\M_1}/\R e_{p_1} \big) \times \widetilde L_{\M_2} \cong \widetilde L_{\M_1}\times \widetilde L_{\M_2}/\R(e_{p_1},0) \ ,
\]
the composite of which clearly maps $\mathbf 1$ to the class of $(\mathbf 1,\mathbf 1)$.

Now assume neither $p_1$ nor $p_2$ are coloops in $\M_1$ and $\M_2$, respectively. As in the definition of the parallel connection, we may assume that $p\coloneqq p_1=p_2$, that $E_1\cap E_2=\{p\}$, and that $E\coloneqq E(\M)=E_1\cup E_2$. 
If $w\in \R^{E}$ is contained in all sets $H_C$, where $C$ ranges over all circuits of $\M$ of type (1) and (2) in the definition of the parallel connections, then it is automatically contained in $H_C$ for all circuits of $\M$ of type (3) as well. This happens if and only if the vector $w|_{E_1}$ obtained by forgetting the coordinates corresponding to $E_2\setminus \{p\}$ is in $\widetilde L_{\M_1}$, and similarly $w|_{E_2}\in \widetilde L_{\M_2}$. In other words, the morphism $\R^{E}\to \R^{E_1}\times \R^{E_2}\cong \R^{E_1\sqcup E_2}$ induced by the canonical map $E_1\sqcup E_2\to E_1\cup E_2$ maps $\widetilde L_\M$ isomorphically onto the conical rational polyhedral subset of $\widetilde L_1\times \widetilde L_2$ where the two coordinates corresponding to $p$ coincide. This induces an isomorphism
\[
\widetilde L_\M\to \widetilde L_{\M_1}\times \widetilde L_{\M_2} /\R(\mathbf 1,0) \ ,
\]
which clearly maps $\mathbf 1$ to the class of $(\mathbf 1,\mathbf 1)$.

For the ``in particular'' statement, we note that if $L_\M\cong L_{\M_1\oplus \M_2}/L$ for some one-dimensional subspace $L$ of $\lineal(L_{\M_1\oplus \M_2})$, then $L_\M\cong \widetilde L_{\M_1}\times \widetilde L_{\M_2}/H$ for some two-dimensional subspace $H$ of $\lineal(\widetilde L_{\M_1}\times \widetilde L_{\M_2})$. But $\widetilde L_{\M_1}\times \widetilde L_{\M_2}/H$ is isomorphic to any quotient $\widetilde L_{\M_1}\times \widetilde L_{\M_2}/H'$ by any two-dimensional subspace $H'$ of $\lineal(\widetilde L_{\M_1}\times \widetilde L_{\M_2})$. In particular, we may take $H'=\Lin\{(\mathbf 1,0), (0,\mathbf 1)\}$, in which case
\[
\widetilde L_{\M_1}\times \widetilde L_{\M_2}/H' \cong L_{\M_1}\times L_{\M_2} \ .
\]
\end{proof}

\begin{lem}
\label{lem:linear if and only quotient linear}
Let $P$ be a conical rational polyhedral subset of $N_\R$ with lineality space $L$. Then $P$ is a tropical linear space if and only if $P/L$ is a tropical linear space. 
\end{lem}

\begin{proof}
For the ``if'' direction we note that $P\cong L\times \left(P/L \right)$. So if $P/L\cong B(\M)$ for some loop-free matroid $\M$ and $r=\dim L$ then it follows from Lemma \ref{lem:parallel connection} that $P$ is isomorphic to the linear space associated to a parallel connection of $\M$ and the uniform matroid $U^{r+1}_{r+1}$ of rank $r+1$ on $r+1$ elements (with respect to any choice of base points). 

For the ``only if'' assume there exists a loop-free matroid $\M$ such that the quotient $L_\M/\lineal(L_\M)$ is \emph{not} a tropical linear space. We may assume that $\M$ has minimal rank among the loop-free matroids with this property. It follows that $\lineal(L_\M)\neq 0$, which is equivalent to $\M$ being disconnected by \cite[Lemma 2.3]{FR10}. Therefore, we can decompose $\M$ as $\M=\M_1\oplus \M_2$ for suitable loop-free matroids $\M_1$ and $\M_2$. Let $\M'$ be the parallel connection of $\M_1$ and $\M_2$ with respect to any choice of base points. Then $\M'$ is a loop-free matroid, its rank is equal to $r(\M)-1$, and $L_{\M'}$ is a quotient of $L_\M$ by  a $1$-dimensional subspace of $\lineal(L_\M)$ by Lemma \ref{lem:parallel connection}. In particular, $L_\M/\lineal(L_\M)\cong L_{\M'}/\lineal(L_{\M'})$, contradicting the minimality property of $\M$.
\end{proof}

\begin{lem}
\label{lem:star of linear is linear}
Let $\Sigma$ be a fan in $N_\R$ such that $|\Sigma|$ is a tropical linear space, and let $\sigma\in\Sigma$. Then $|\Star_\Sigma(\sigma)|$ is a tropical linear space as well. 
\end{lem}

\begin{proof}
We may assume that $|\Sigma|=L_\M$ for some loop-free matroid $\M$. Let $w\in\relint(\sigma)$. By \cite[Lemma 2.2]{FR10} the local cone $\LC_w|\Sigma|$ is the tropical linear space associated to the matroid $\M_w$ whose bases are the bases of $\M$ with minimal $w$-weight. The set $|\Star_\Sigma(\sigma)|$ is the quotient of $\LC_w|\Sigma|$ by $\Lin(\sigma)$. Therefore, the quotients of $\LC_w|\Sigma|$ and $|\Star_\Sigma(\sigma)|$ by their lineality spaces are isomorphic. Using Lemma \ref{lem:linear if and only quotient linear}, we conclude that $|\Star_\Sigma(\sigma)|$ is a tropical linear space.
\end{proof}

\subsection{Poincar\'e duality rings and blow-ups}

We first recall the definition of a Poincar\'e duality ring.

\begin{defn}
A commutative $\N$-graded ring $R^*$ is a  \emph{Poincar\'e duality ring of dimension $d$} if 
\begin{enumerate}[label= \alph*)]
\item $R^*$ is finitely generated as a $\Z$-module.
\item $R^d\cong \Z$ and $R^k=0$ for all integers $k>d$,
\item for every integer $k\in\Z$ the multiplication map and any identification $R^d\cong\Z$ induce an isomorphism
\[
R^k\cong \Hom(R^{d-k},\Z) \ .
\]
\end{enumerate}
\end{defn}

The following result is one of the key ingredients in our proof of Theorem \ref{thm:duality for fans}. It is a slight generalization of \cite[Proposition 2.2]{Petersen16} (with a similar proof), and of a purely algebro-geometric nature. 

\begin{prop}
\label{prop:blow-up is duality ring}
Let $Z$ be a smooth subvariety of codimension $c>1$ of a smooth variety $X$, and let $Y$ be the blow-up of $X$ along $Z$. Assume that for some positive integer $d$ the following three conditions on the Chow rings of $X$ and $Z$ are satisfied: 
\begin{enumerate}
\item $A^k(X) = 0$ for all integers $k>d$,
\item $A^k(Z)=0$ for all integers $k>d-c$,
\item The morphism
$
\Hom_\Z(A^d(X),\Z)\to \Hom_\Z(A^{d-c}(Z),\Z)
$
induced by the push-forward is an isomorphism.
\end{enumerate}
 Then $A^*(Y)$ is a Poincar\'e duality ring of dimension $d$ if and only if $A^*(X)$ is a Poincar\'e duality ring of dimension $d$ and $A^*(Z)$ is a Poincar\'e duality ring of dimension $d-c$.
\end{prop}

\begin{proof}
Consider the diagram 
\begin{center}
\begin{tikzpicture}
\matrix[matrix of math nodes, row sep= 5ex, column sep= 4em, text height=1.5ex, text depth= .25ex]{
|(E)|  E		& 
|(Y)|  Y	\\
 |(Z)| Z 	&
|(X)|	X\\
};
\begin{scope}[->,font=\footnotesize,auto]
\draw (E) --node{$j$} (Y);
\draw (Z) --node{$i$} (X);
\draw (E)--node{$g$} (Z);
\draw (Y)--node{$f$} (X);
\end{scope}
\end{tikzpicture}
\end{center}
where $E$ is the exceptional divisor. Combining \cite[Theorem 3.3]{F98} and \cite[Proposition 6.7]{F98} we see that, for all $k\in \Z$, there is an isomorphism
\[
\left(\bigoplus_{i=1}^{c-1} A^{k-i}(Z) \right)\oplus A^k(X) \to A^k(Y)
\]
that maps $x\in A^k(X)$ to $f^*(x)$ and $z\in A^{k-i}(Z)$ to $j_*(g^*z)\cdot E^{i-1}$. In particular $A^*(Y)$ is free if and only if both $A^*(X)$ and $A^*(Z)$ are free. So if either $A^*(Y)$ or both $A^*(X)$ and $A^*(Z)$ are Poincar\'e duality rings, then the condition of 
\[
\Hom_\Z(A^d(X),\Z)\to \Hom_\Z(A^{d-c}(Z),\Z)
\]
being an isomorphism is equivalent to the push-forward
\[
A^{d-c}(Z)\to A^d(X)
\]
being an isomorphism. In this case, a combination of \cite[Theorem 3.3]{F98} and \cite[Proposition 6.7]{F98} shows that the morphisms
\begin{align*}
A^d(X)\to A^d(Y),&\;\; x\mapsto f^*(x)  \\
A^{d-c}(Z)\to A^d(Y),& \; \; z\mapsto j_*g^*z\cdot E^{c-1}
\end{align*}
are isomorphism, so we can identify $A^d(X)$, $A^d(Y)$, and $ A^{d-c}(Z)$ with $\Z$ in a way compatible with these isomorphisms.

Having noted this, we can proceed analogously to the proof of \cite[Proposition 2.2]{Petersen16}: by \cite[Example 8.3.9]{F98} the product in $A^*(Y)$ of
\[
\sum_{i=1}^{c-1} j_*g^*(z_i)E^{i-1} + f^*(x) \in A^k(Y) \ ,
\]
with $z_i\in A^{k-i}(Z)$ and $x\in A^k(X)$, and
\[
\sum_{i=1}^{c-1} j_*g^*(z'_i)E^{i-1} + f^*(x') \in A^{d-k}(Y) \ ,
\]
with $z'_i\in A^{d-k-i}(Z)$ and $x'\in A^{d-k}(X)$, is given  by
\begin{multline*}
\sum_{i,j=1}^{c-1} j_*g^*(z_i\cdot z'_j) \cdot E^{i+j-i} 
+ \sum_{i=1}^{c-1}j_*g^*(z_i\cdot i^*x') 
+\sum_{i=1}^{c-1}j_*g^*(z'_i\cdot i^*x) 
+f^*(x\cdot x')
=\\
=
\sum_{i+j\geq c}^{c-1} j_*g^*(z_i\cdot z'_j) \cdot E^{i+j-i} 
+f^*(x\cdot x')
\end{multline*}
It follows that, with the appropriate choice of bases, the pairing 
\[
A^k(Y)\otimes A^{d-k}(Y)\to A^d(Y)\cong \Z
\]
can be expressed as a block-diagonal  matrix, whose blocks are the matrices of the pairings 
\begin{align*}
A^k(X)\otimes A^{d-k}(X)&\to A^d(X)\cong \Z  \ , \text{ and} \\
A^{k-i}\otimes A^{d-c-(k-i)}(Z)&\to A^{d-c}(Z)\cong \Z \ ,
\end{align*}
where $1\leq i\leq c-1$. Since block-diagonal matrices are invertible if and only if all the blocks are invertible, this finishes the proof.
\end{proof}

\begin{rem}
We only need Proposition \ref{prop:blow-up is duality ring} in the case where $X$ is a toric variety and $Z$ is torus-invariant, in which case the statement can be phrased purely combinatorially. It would be interesting to give a proof by showing the decomposition of the Chow ring of the blow-up combinatorially. This seems to be a challenging problem though, and relying on a more general principle from algebraic geometry provides a clean and geometrically meaningful way around it.
\end{rem}

\subsection{Duality for toric varieties}

The next result relates the Chow rings of a smooth fan, its stars, and its star subdivisions.

\begin{prop}
\label{prop:stellar subdivisions}
Let $\Sigma$ be a smooth rational polyhedral fan such that $|\Sigma|$ is a tropical linear space, and let $\sigma\in \Sigma$. Then $R^*(\Sigma^{\star}(\sigma))$ is a Poincar\'e duality ring if and only if $R^*(\Sigma)$ and $R^*(\Star_\Sigma(\sigma))$ are Poincar\'e duality rings.
\end{prop}

\begin{proof}
Let $d=\dim\Sigma$ and $c=\dim\sigma$. Then $R^k(\Sigma)\cong R^k(\Sigma^{\star}(\sigma))\cong 0$ for all integers $k>d$ and $R^k(\Star_\Sigma(\sigma))=0$ for all integers $k>d-c$. Furthermore, $\Hom(R^d(\Sigma),\Z)$ is the group of $d$-dimensional Minkowski weights on $\Sigma$, which is isomorphic to $\Z$ by \cite[Lemma 2.4]{FR10} because $|\Sigma|$ is a tropical linear space. In particular, $R^d(\Sigma)\neq 0$, and similarly, $R^d(\Sigma^{\star}(\sigma))\neq 0$ and $R^{d-c}(\Star_\Sigma(\sigma))\neq 0$, the latter statement using Lemma \ref{lem:star of linear is linear}. We conclude that if $R^*(\Sigma)$, $R^*(\Sigma^{\star}(\sigma))$, or $R^*(\Star_\Sigma(\sigma))$ are Poincar\'e duality rings, then they are Poincar\'e duality rings of dimensions are $d$, $d$, and $d-c$, respectively.

Let $X=X_\Sigma$ and $X'=X_{\Sigma^{\star}(\sigma)}$ be the toric varieties (over an algebraically closed field) associated to $\Sigma$ and $\Sigma^{\star}(\sigma)$, respectively, and let $Z\subseteq X$ be the torus-invariant closed subvariety corresponding to $\sigma$ by the orbit-cone correspondence. Then $Z$ is naturally isomorphic to the toric variety associated to $\Star_\Sigma(\sigma)$ (and, in particular, is smooth), and $X'$ is the blow-up of $X$ at $Z$. As explained in \S\ref{subsec:chow rings of fans}, there are natural isomorphisms $R^*(\Sigma)\to A^*(X)$, $R^*(\Sigma^{\star}(\sigma))\to A^*(X')$, and $R^*(\Star_\Sigma(\sigma))\to A^*(Z)$. Using Proposition \ref{prop:blow-up is duality ring}, this reduces the assertion to proving that the composite 
\begin{multline}
\label{equ:dual of push-forward}
\Hom(R^d(\Sigma),\Z)\cong\Hom(A^d(X),\Z)\to \\
\to\Hom(A^{d-c}(Z),\Z)\cong \Hom(R^{d-c}(\Star_\Sigma(\sigma)),\Z)
\end{multline}
induced by the composite $R^{d-c}(\Star_\Sigma(\sigma)) \cong A^{d-c}(Z)\to A^{d}(X)\cong R^d(\Sigma)$, where the morphisms in the middle is the push-forward of cycles, is an isomorphism. Note that this morphism is easy to describe on square-free monomials. If  $\tau\in\Sigma(d)$ contains $\sigma$, then the monomial $x_{\tau/\sigma}\in R^d(\Star_\Sigma(\sigma))$ maps to the torus invariant cycle $[V(\tau/\sigma)]\in A^{d-c}(Z)$, which is mapped to $[V(\tau)]\in A^{d}(X)$ under the push-forward morphism. This, in turn, is mapped to $x_\tau\in R^d(\Sigma)$. By \cite[Lemma 2.4]{FR10}, the group $\Hom(R^d(\Sigma),\Z)$ of $d$-dimensional Minkowski weights on $\Sigma$ is isomorphic to $\Z$, and it is generated by the morphism $R^d(\Sigma)\to \Z$ mapping the monomial $x_\tau$ corresponding to $\tau\in\Sigma(d)$ to $1$. By what we just saw, the composite morphism displayed in (\ref{equ:dual of push-forward}) maps this to the morphism $R^d(\Star_\Sigma(\sigma))\to \Z$ that sends a monomial $x_{\tau/\sigma}$ corresponding to a cone $\tau\in \Sigma(d)$ containing $\sigma$ to $1$. Again by \cite[Lemma 2.4]{FR10}, this morphism freely generates the group $\Hom_\Z(R^{d-c}(\Star_\Sigma(\sigma)),\Z)$ of $(d-c)$-dimensional Minkowski weights on $\Star_\Sigma(\sigma)$. Therefore, the composite morphism displayed in (\ref{equ:dual of push-forward}) is indeed an isomorphism, finishing the proof.
\end{proof}

We are now ready to prove Theorem \ref{thm:duality for fans}. We have restated it here in terms of Chow rings of smooth fans instead of Chow rings of toric varieties.

\begin{thm}
\label{thm:duality for fans2}
Let $\Sigma$ be a smooth rational polyhedral fan in $N_\R$ such that $|\Sigma|$ is a tropical linear space. Then its Chow ring $R^*(\Sigma)$ is a Poincar\'e duality ring. 
\end{thm}

\begin{proof}
We do induction on $d=\dim(\Sigma)$, the base case $d=0$ being trivial. If $d>0$, we may assume that $|\Sigma|$ is supported on $L_\M$ for some loopless matroid $\M$. By \cite[Theorem 6.19]{AHK15}, the statement is true for the fine subdivision fan structure $\Sigma_\M$ on $L_\M$. By the weak factorization theorem for toric varieties \cite[Theorem A]{wj}, there is a sequence of star subdivisions and star assemblings (the inverses of star subdivisions) connecting $\Sigma$ and $\Sigma_\M$. We thus reduce to showing that if $\sigma\in\Sigma$, then $R^*(\Sigma^{\star}(\sigma))$ is a Poincar\'e duality ring if and only if $R^*(\Sigma)$ is a Poincar\'e duality ring. By Proposition \ref{prop:stellar subdivisions} and the induction hypothesis, it suffices to show that $|\Star_\Sigma(\sigma)|$ is a tropical linear space. This is true by Lemma \ref{lem:star of linear is linear}, finishing the proof.	
\end{proof}

\section{Poincar\'e duality on tropical manifolds}

\subsection{Tropical fan cycles and piecewise polynomial functions}
\label{subsec:fan cycles}
If $\Sigma$ is a smooth fan in $N_\R$ and $\Delta$ is a smooth proper subdivision of $\Sigma$, then the induced proper morphism $X_\Delta\to X_\Sigma$ induces a morphism of abelian groups $f\colon R^*(\Delta)\to R^*(\Sigma)$, and a morphism of rings $g\colon R^*(\Sigma)\to R^*(\Delta)$, the first corresponding to the push-forward and the latter to the pull-back. Both of them have combinatorial descriptions.

We start with a description of $f\colon R^*(\Delta)\to R^*(\Sigma)$. If $\delta\in \Delta(k)$, and $\sigma$ is the minimal cone of $\Sigma$ containing $\delta$, then 
\[
f(x_\delta)=\begin{cases}
x_\sigma\ , &\text{if }\dim(\sigma)=k \ ,\\
0 \ ,		& \text{else.}
\end{cases}
\]
The induced morphism 
\[
\Mink_*(\Sigma)\cong\Hom(R^*(\Sigma),\Z) \to \Hom(R^*(\Delta),\Z)\cong \Mink_*(\Delta)
\]
is given by the refinement of weighted fans (cf.\ \cite[Definition 2.8]{AR10}). Passing to the direct limit over all smooth fans with a given conical rational polyhedral support $P$, we obtain the \emph{group of tropical fan cycles} introduced in \cite[Definition 2.12]{AR10}:
\[
Z_*^{\mathrm{fan}}(P)\coloneqq\injlim_{\Sigma:|\Sigma|=P} \Mink_*(\Sigma)=\injlim_{\Sigma: |\Sigma|=P} \Hom_\Z(R^*(\Sigma),\Z) \ ,
\]
where the direct limit is taken over all smooth fan structures for $P$.

If $P$ is a tropical linear space, then $Z^{\mathrm{fan}}_{\dim P}(P)\cong \Z$ is freely generated by the cycle $[P]$ represented by Minkowski weights having value $1$ on all top-dimensional cones. Moreover, there is a tropical intersection product in this case, making $Z_*^{\mathrm{fan}}(P)$ into a ring with unity $[P]$ (see \cite{FR10,Shaw13}).

The morphism $g\colon R^*(\Sigma)\to R^*(\Delta)$ is more conveniently described using piecewise polynomial functions. It is the morphism $\PP^*_\Z(\Sigma)/\LPP_\Z(\Sigma)\to \PP^*_\Z(\Delta)/\LPP_\Z(\Delta)$ induced by the inclusion $\PP^*_\Z(\Sigma)\hookrightarrow \PP^*_\Z(\Delta)$. We can again pass to the direct limit over all smooth fans with a given conical rational polyhedral support $P$ and obtain the ring of piecewise polynomial functions
\[
\PP_\Z^*(P)=\injlim_{\Sigma\colon |\Sigma|=P} \PP_\Z(\Sigma) \ .
\]
If $\LPP_\Z(P)$ denotes the ideal in $\PP_\Z^*(P)$ generated by the linear functions, then
\[
\PP_\Z^*(P)/\LPP_\Z(P)\cong \injlim_{\Sigma:|\Sigma|=P} \PP^*_\Z(\Sigma)/\LPP_\Z(\Sigma)=\injlim_{\Sigma:|\Sigma|=P} R^*(\Sigma) \  .
\] 
The ring $R^*(\Sigma)$, and hence also the ring $\PP_\Z^*(\Sigma)$, acts on $\Mink_*(\Sigma)=\Hom_\Z(R^*(\Sigma),\Z)$ via precomposition with the multiplication. In degree $1$, every element in $\PP_\Z^1(\Sigma)$ is a \emph{tropical rational function} in the sense of \cite{AR10}. There also is an intersection pairing of tropical rational functions with tropical fan cycles (see \cite[Definition 3.4]{AR10}).

\begin{lem}
\label{lem:the two products agree}
Let $\Sigma$ be a smooth fan. Then the pairing $\PP_\Z^1(\Sigma)\otimes_\Z \Mink_*(\Sigma)\to \Mink_*(\Sigma)$ from above coincides with the intersection pairing of tropical rational functions with tropical cycles. 
\end{lem}

\begin{proof}
Let $c\in \Mink_k(\Sigma)$ be a Minkowski weight, and let $\phi\in \PP_\Z^1(\Sigma)$ be a piecewise linear function on $\Sigma$. We write $c(x_\sigma)=c(\sigma)$ when identifying $\Mink_k(\Sigma)$ with $\Hom_\Z(R^k(\Sigma),\Z)$. Let $\tau\in \Sigma(k-1)$. We will compare the weights on $\tau$ of the two pairings of $\phi$ and $\tau$. Because both are defined modulo integral linear functions, we may assume that $\phi\vert_\tau=0$ after potentially subtracting from $\phi$ an integral linear function that agrees with $\phi$ on $\tau$.
Let ``$\cdot$'' denote the action defined via the multiplication on $R^*(\Sigma)$.  The Minkowski weight $\phi\cdot c$ has at $\tau$ the weight 
\[
\phi\cdot c(\tau)= c\left(x_\tau\cdot \sum_{\rho\in\Sigma(1)\setminus\tau(1)} \phi(u_\rho) \cdot x_\rho\right)=\sum_{\rho\in\Sigma(1)\setminus\tau(1)} \phi(u_\rho)\cdot x_\rho \cdot x_\tau \ .
\]
If  $\tau$ and $\rho\in \Sigma(1)\setminus \tau(1)$ span a $k$-dimensional cone $\sigma$, then $x_\tau\cdot x_\rho=x_\sigma$ and $u_\rho$ is a representative of the lattice normal vector of $\sigma$ with respect to $\tau$. Otherwise, $x_\tau\cdot x_\rho=0$. It follows that 
\[
\phi\cdot c(\tau)= \sum_{\sigma\in\Sigma(k)\colon \tau\subseteq \sigma} \phi(u_{\sigma/\tau})x_\sigma \ ,
\]
which is precisely the weight on $\tau$ of the intersection pairing of the tropical rational function $\phi$ and $c$ considered in \cite{AR10}.
\end{proof}

Since $\PP^*_\Z(\Sigma)$ is generated in degree $1$, Lemma \ref{lem:the two products agree} shows that the action of $\PP^*_\Z(\Sigma)$ on $\Mink_*(\Sigma)$ defined above agrees with with the action considered in \cite{F11}. Because the intersection pairing of \cite{AR10} is compatible with refinements, this also shows that we obtain an induced action of $\PP_\Z^*(P)$ on $Z^{\mathrm{fan}}_*(P)$ for every conical rational polyhedral set $P$.

\subsection{Cycles tropical manifolds}
\label{subsec:cycles on manifolds}

On any subset $V\subseteq N_\R$ there is a natural sheaf $\Aff_V$ of affine linear functions with integer slopes. A \emph{(boundaryless) rational polyhedral space} is a second-countable Hausdorff space $X$ together with a sheaf of continuous functions $\Aff_X$, such that every point $x\in X$ is contained in a \emph{chart}, by which we mean a homeomorphism of an open neighborhood $U$ of $x$ and an open subset $V$ of a conical rational polyhedral set $P$ which identifies $\Aff_U$ with $\Aff_V$. A rational polyhedral space $X$ is called a \emph{( boundaryless) tropical manifold} if $P$ can always be chosen as a tropical linear space for all $x\in X$.

Let $X$ be a rational polyhedral space, let $x\in X$, and let $U\xrightarrow{\phi}V\subseteq P$ be a chart around $x$. Then the isomorphism type of the local cone $\LC_{\phi(x)}(P)$ does not depend on the chart $\phi$. We choose one such isomorphism class as the \emph{local cone} of $X$ at $x$ and we denote it by $\LC_x(X)$. By construction, there exists an open neighborhood of $0$ in $\LC_x(X)$ that is naturally isomorphic to an open neighborhood of $x$ in $X$.

To define tropical cycles on rational polyhedral spaces we first recall that if $P$ is a conical rational polyhedral set in $N_\R$, then any  fan cycle $A\in Z^{\mathrm{fan}}_k(P)$ is defined by an integer valued function on $P$. If $A$ is a represented by a Minkowski weight $c\in \Mink_k(\Sigma)$ for a fan $\Sigma$ with $|\Sigma|=P$, then the support $|A|$ of $A$ is defined as the union of all cones $\sigma\in\Sigma(k)$ such that $c(\sigma)\neq 0$. This is independent of the representative $c$. Let $|A|^{\max}$ denote the set of all $w\in |A|$ that have a neighborhood which is an open subset of a linear subspace of $N_\R$. There exists a unique locally constant function $|A|^{\max}\to \Z$ that has value $c(\sigma)$ on the relative interior $\relint(\sigma)$ of all $\sigma\in\Sigma(k)$ with $\sigma\subseteq \vert A\vert$. Extending this function by $0$, one obtains an integer valued function on $|P|$ that is independent of $c$ and determines $A$. In fact, since this function is invariant under the $\R_{>0}$ scaling action, it is uniquely determined by its germ at the origin. Based on this, we will identify $Z_*^{\mathrm{fan}}(P)$ with germs of integer valued functions at the origin $0\in P$ in the rest of this section.

If $X$ is a rational polyhedral space, we define the group $Z_k(X)$ of \emph{tropical $k$-cycles} on $X$ as the group of integer valued function $A\colon X\to \Z$ for which every germ $A_x$ is a tropical fan $k$-cycle on $\LC_x X$. As this condition is local, the presheaf 
\[
X\supseteq U\mapsto Z_k(U)
\]
is a sheaf, which we will denote by $(\mZ_X)_k$. It follows from the definition that $(\mZ_{X,x})_k\cong Z_k^{\mathrm{fan}}(\LC_x X)$ for all $x\in X$. 

If $X$ is a tropical manifold, $Z_{\dim X}(X)$ is freely generated by the cycle $[X]$ which has value $1$ on a dense open subset of $X$. Moreover the tropical intersection products on the local cones of $X$ induce an intersection product on $Z_{\dim X}(X)$, making it a commutative ring with unity $[X]$. See \cite{FR10,Shaw13} for more details.

\subsection{Cocycles tropical manifolds}
\label{subsec:cocycles on manifolds}

To define piecewise polynomial functions and tropical cocycles on rational polyhedral spaces, we need to allow polynomials with non-integer coefficients. Let $\Sigma$ be a smooth fan in $N_\R$. For each $k\in \N$ we denote by $\PP^k(\Sigma)$ the group of continuous functions $|\Sigma|\to \R$ such that  the restriction to each $\sigma\in \Sigma$ is given by a polynomial function of degree at most $k$, whose degree $k$-part has integer coefficients. We remark that this makes our definition different from the one used in \cite{F11}.

The direct sum
\[
\PP^*(\Sigma)\coloneqq \bigoplus_{k\in \Z} \PP^k(\Sigma)
\]
is naturally a graded ring. Let $\LPP(\Sigma)$ be the ideal in $\PP(\Sigma)$ generated in degree $1$ by the affine linear functions with integer slopes. It is straightforward to check that $\PP^*(\Sigma)/\LPP(\Sigma)\cong \PP^*_\Z(\Sigma)/\LPP_\Z(\Sigma)$, and in particular $\PP^*(\Sigma)/\LPP(\Sigma)\cong R^*(\Sigma)$. For a conical polyhedral set $P$ we define
\[
\PP^{*}(P)=\injlim_{\Sigma\colon |\Sigma|=P} \PP^*(\Sigma) \ .
\]
In other words, $\PP^k(P)$ is the set of functions $P\to \R$ that are in $\PP^k(\Sigma)$ for some smooth fan structure $\Sigma$ for $P$. 

\begin{rem}
\label{rem:inconsitence in definition of cocycles}
The advantage of considering $\PP^*(P)$ in favor of $\PP_\Z^*(P)$ is that for $\phi\in \PP^k(P)$ and $p\in P$ the germ $\phi_p$ can be identified with the germ of an element of $\PP^k(\LC_p P)$, which is not true in general for an element of $\PP_\Z^k(P)$ because the expansion of a homogeneous polynomial with integer coefficients around a point with non-integer coordinates will in general neither be homogeneous, nor have integer coefficients. This distinction has not been made in \cite{F11}, leading to a minor inaccuracies in the definition of tropical cocycles in \textit{loc.\ cit.}
\end{rem}

If $X$ is a rational polyhedral space we define $\PP^{k}(X)$ as the group of functions $\phi\colon X\to \R$ such that for each $x\in X$ the germ $\phi_x$ can be identified with the germ of an element of $\PP^{k}(\LC_x X)$, and $\PP^*(X)=\bigoplus_{k\in \N}\PP^k(X)$. Since this is a local definition, the presheaf
\[
X\supseteq U\mapsto \PP^*(U)
\]
is a sheaf of graded rings, which we denote by $\mPP^*_X$. By construction, the stalk $\mPP^*_{X,x}$ at $x\in X$ is naturally isomorphic to $\PP^*(\LC_x X)$.  The sheaf $\Aff_X$ of affine linear functions on $X$ is a subsheaf of $\mPP^{1}_X$, and we denote the ideal sheaf generated by $\Aff_X$ by $\mLPP_X$. 

\begin{defn}
Let $X$ be a rational polyhedral space, then we define the sheaf of \emph{tropical cocycles} as
\[
\mathcal C^*_X\coloneqq \mPP^*_X/\mLPP_X \ .
\]
Its ring of global sections
\[
C^*(X)\coloneqq \Gamma(X,\mathcal C^*_X)
\]
is the ring of \emph{tropical cocycles} on $X$.
\end{defn}

By construction, the stalk $\mathcal C^*_{X,x}$ at $x\in X$ is isomorphic to $\PP^*(\LC_x X)/\LPP(\LC_x X)\cong \PP^*_\Z(\LC_x X)/\LPP_\Z(\LC_x X)$ and hence $(\mZ_{X,x})_*$ is an $\mC^*_{X,x}$-module. It is easy to check that this makes $(\mZ_X)_*$ a sheaf of $\mC^*_X$-modules. Indeed, since $\mC^*_{X,x}$ is generated in degree $1$, it suffices to show that the pairings $\mC^1_{X,x}\otimes_\Z (\mZ_{X,x})_*\to (\mZ_{X,x})_*$  induce a pairing $\mC^1_X\otimes_\Z (\mZ_X)_*\to (\mZ_X)_*$. The sections in $\mC^1_X$ are usually called \emph{tropical Cartier divisors}. The intersection pairing of tropical Cartier divisors and tropical cycles is local by \cite[Proposition 1.1]{Rau16}, so then the statement follows from Lemma \ref{lem:the two products agree}.

\subsection{Duality for tropical manifolds}

To deduce Theorem \ref{thm:duality for manifolds} from Theorem \ref{thm:duality for fans2}, we will need the following lemma.

\begin{lem}
\label{lem:local isomorphism}
Let $L$ be a tropical linear space. Then the morphism
\[
\PP^*(L)/\LPP(L)\to Z^{\mathrm{fan}}_*(L),\;\; \alpha \mapsto \alpha\cdot [L]
\]
is an isomorphism. 
\end{lem}

\begin{proof}
As discussed above, we have
\begin{align*}
\PP^*(L)/\LPP(L)&\cong \injlim_{\Sigma:|\Sigma|=L} R^*(\Sigma) \ ,	 \\
Z^{\mathrm{fan}}_*(L) &\cong \injlim_{\Sigma:|\Sigma|=L} \Hom_\Z(R^*(\Sigma),\Z)\ ,
\end{align*}
where the direct limit is over all smooth fan structures $\Sigma$ for $L$. 

The action of $\injlim_{\Sigma:|\Sigma|=L} R^*(\Sigma)$ on $\injlim_{\Sigma:|\Sigma|=L} \Hom_\Z(R^*(\Sigma),\Z)$ is induced by the compatible actions for the fans $\Sigma$ appearing in the direct limit, of $R^*(\Sigma)$ on $\Hom_\Z(R^*(\Sigma),\Z)$ defined by precomposition with the multiplication map. For each fan $\Sigma$ appearing in the  direct limit, the class $[L]$ is represented by the morphism $[L]_\Sigma\colon: R^{\dim L}(\Sigma) \to \Z$ sending $x_\sigma$ to $1$ for each $\sigma\in \Sigma(\dim L)$. By Theorem \ref{thm:duality for fans2}, $[L]_\Sigma$ is an isomorphism, and another application of Theorem \ref{thm:duality for fans2} shows that morphism
\[
R^*(\Sigma) \to \Hom_\Z(R^*(\Sigma),\Z),\;\; \phi\mapsto\phi\cdot [L]_\Sigma
\]
is an isomorphism. Since multiplication with $[L]_\Sigma$ is an isomorphism for all fans $\Sigma$ appearing in the direct limit, it follows immediately that multiplication with $[L]$ is an isomorphism as well, finishing the proof.
\end{proof}

\begin{thm}
\label{thm:duality for manifolds2}
Let $X$ be a tropical manifold. Then the morphism
\[
C^*(X)\to Z_*(X),\;\; \alpha\mapsto \alpha\cdot [X]
\]
is an isomorphism of rings. 
\end{thm}

\begin{proof}
The morphism is defined locally, that is it is obtained by taking the global section of the morphism of sheaves
\begin{align*}
\mathcal C^*_X&\to  (\mathcal Z_X)_* \\
C^*(U)\ni \alpha &\mapsto \alpha \cdot [U] \in Z_*(U) \ .
\end{align*}
It thus suffices to show that this is an isomorphism on stalks, that is that for each $x\in X$ the induced morphism
\[
\mathcal C^*_{X,x}\to (\mathcal Z_{X,x})_*
\]
is an isomorphism. But this morphism can be naturally identified with the morphism
\[
\PP^*(\LC_x X)/\LPP(\LC_x X) \to Z_*(\LC_x X),\;\; \alpha \mapsto \alpha\cdot [\LC_x X] \ ,
\]
which is an isomorphism by Lemma \ref{lem:local isomorphism}. To finish the proof, we need to show that the morphism $\mathcal C^*_X \to (\mathcal Z_X)_*$ from above is a morphism of sheaves of rings, where the multiplication on $(\mathcal Z_X)_*$ is the tropical intersection product. Since $C^*_X$ is generated in degree $1$, it suffices to show that tropical intersection products respect intersections with piecewise linear functions, which has been done in \cite[Theorem 4.5]{FR10}.
\end{proof}

\begin{rem}
If one thinks of the elements of $C^*(X)$ as cohomology classes on $X$, of the elements of $Z_*(X)$ as homology classes on $X$, and of the product $\alpha\cdot [X]$ as the cap product of a cohomology class with the fundamental class on $X$, then the statement of Theorem \ref{thm:duality for manifolds2} is the direct analogue of classical Poincaré duality.
\end{rem}

We can use Theorem \ref{thm:duality for manifolds2} to define pull-back morphisms $f^*\colon Z_*(Y) \to Z_*(\vert A\vert)$ whenever $f\colon \vert A\vert \to Y$ is a morphism from the support of a tropical cycle $A$ to a tropical manifold $Y$. For this we note that a morphism $f\colon X\to Y$ between two rational polyhedral spaces is a continuous map that pulls back functions in $\Aff_Y$ to functions in $\Aff_X$. Any such morphism pulls back piecewise polynomial function on $Y$ to piecewise polynomial functions on $X$, so we obtain a pull-back morphism $f^{-1}\mathcal C_Y^*\to \mathcal C_X^*$ of sheaves of rings, which defines a ring morphism $f^*\colon C^*(Y)\to C^*(X)$ between the rings of tropical cocycles.

\begin{cor}
\label{cor:pullbacks of tropical cycles}
Let $f\colon X\to Y$ be a morphism from a rational polyhedral space $X$ to a tropical manifold $Y$, and let $A\in Z_*(X)$. Then there exists a unique pull-back morphism
\[
f^*\colon Z_*(Y)\to Z_*(X)
\]
that satisfies $f^*[Y]=A$ and
\[
f^*(c\cdot B)= f^*(c)\cdot f^*(B)
\]
for every $c\in C^*(Y)$ and $B\in Z_*(Y)$.
\end{cor}

\begin{proof}
First note that if $f^*\colon Z_*(Y)\to Z_*(X)$ satisfies the given conditions, then 
\[
f^*(c\cdot [Y])= f^*(c)\cdot f^*[Y]
\] 
for all $c\in C^*(Y)$. Since $C^*(Y)\ni c \mapsto c\cdot [Y] \in Z_*(Y)$ is surjective by Theorem \ref{thm:duality for manifolds2}, this shows that $f^*$ is unique. 

For existence, we can define $f^*$ as the composite
\[
Z_*(Y)\to C^*(Y)\xrightarrow{f^*} C^*(X)\xrightarrow{c\mapsto c\cdot A} Z_*(X) \ ,
\] 
where the first morphism is the inverse of $C^*(Y)\ni c \mapsto c\cdot [Y] \in Z_*(Y)$, which exists by Theorem \ref{thm:duality for manifolds2}. It follows directly from this definition that $f^*[Y]= A$. To check that it also satisfies the second condition, let $c\in C^*(Y)$ and $B\in Z_*(Y)$. Let $d\in C^*(Y)$ be the unique cocycle with $d\cdot [Y]= B$. Then we have 
\[
f^*(c\cdot B)=f^*\left(c\cdot (d\cdot [Y])\right)= f^*(c\cdot d)\cdot A= f^*(c)\cdot (f^*(d)\cdot A)=f^*(c)\cdot f^*(B) \ ,
\]
finishing the proof.
\end{proof}

\bibliography{Bib}

\bibliographystyle{amsalpha}

\end{document}